\documentclass[11pt,a4paper,reqno]{amsart}
\usepackage{amscd,amssymb}
\usepackage[arrow,matrix]{xy}
\usepackage{xcolor}
\usepackage[colorlinks=true,linkcolor=blue,citecolor=purple]{hyperref}
\usepackage{graphicx}
\usepackage{fullpage}
\usepackage{mathrsfs}

\theoremstyle{plain}
\newtheorem{thm}[subsection]{Theorem}
\newtheorem{lem}[subsection]{Lemma}

\newtheorem{cor}[subsection]{Corollary}

\theoremstyle{definition}

\newtheorem{ex}[subsection]{Example}

\theoremstyle{remark} 
\newtheorem{rem}{Remark} 

\def\deg{{\mathrm{deg}}}
\def\KK{{\mathbb{K}}}
\def\AA{{\mathbb{A}}}
\def\PP{{\mathbb{P}}}
\def\UU{{\mathbb{U}}}
\def\ZZ{{\mathbb{Z}}}
\def\QQ{{\mathbb{Q}}}
\def\Res{{\mathrm{Res}}}

\def\Disc{{\mathrm{Disc}}}
\def\Spec{{\mathrm{Spec}}}

\def\Rc{{\mathcal{R}}}
\def\Ec{{\mathcal{E}}}
\def\Dc{{\mathcal{D}}}

\begin{document}

\title{Formulas for the eigendiscriminants of ternary and quaternary forms}
\author{Laurent Bus\'e}
\address{Universit\'e C\^ote d'Azur, Inria, 2004 route des Lucioles, 06902 Sophia Antipolis, France.}
\email{laurent.buse@inria.fr}
\date{\today}



\begin{abstract}
A $d$-dimensional tensor $A$ of format $n\times n\times \cdots \times n$ defines naturally a rational map $\Psi$ from the projective space $\PP^{n-1}$ to itself and its eigenscheme is then the subscheme of $\PP^{n-1}$ of fixed points of $\Psi$. The eigendiscriminant is an irreducible polynomial in the coefficients of $A$ that vanishes for a given tensor if and only if its eigenscheme is singular. In this paper, we contribute two formulas for the computation of eigendiscriminants in the cases $n=3$ and $n=4$. In particular, by restriction to symmetric tensors, we obtain closed formulas for the eigendiscriminants of plane curves and surfaces in $\PP^3$ as the ratio of some determinants of resultant matrices.
\end{abstract}

\maketitle

\section{Introduction}\label{sec:intro}

The geometric counterpart of a square matrix $A$ of size $n\times n$ with entries in a field $\KK$ is a projective linear rational map 
\begin{eqnarray*}
	\Psi : \PP^{n-1}_\KK & \dasharrow & \PP^{n-1}_\KK \\
	(x_0:\cdots: x_{n-1}) & \mapsto & (\psi_{0}:\cdots\psi_{n-1}),
\end{eqnarray*}
where the linear forms $\psi_0,\ldots,\psi_{n-1}$ are obtained by right-multiplying $A$ with the column-vector $(x_0,\ldots,x_{n-1})^T$.
The eigenvectors of $A$ then correspond to the scheme of fixed points of $\Psi$, called the \emph{eigenscheme} of $A$, which is the subscheme of $\PP^{n-1}$ defined by the $2$-minors of the matrix
\begin{equation}\label{eq:2minors}
\left( 
\begin{array}{cccc}
	\psi_0 & \psi_1 & \cdots & \psi_{n-1} \\
	x_0 & x_1 & \cdots & x_{n-1}
\end{array}
\right).
\end{equation}
The extension to tensors is a topic that received recently a lot of interest, in particular with the development of algebraic techniques for studying the spectral theory of tensors (see \cite{ASS17} and the references therein). More precisely, if A is a $d$-dimensional tensor of format $n\times n\times\cdots \times n$, it also defines a rational map $\Psi$ from $\PP^{n-1}$ to itself but the polynomials $\psi_0,\ldots,\psi_{n-1}$ are now homogeneous polynomials of degree $d-1$ in $\KK[x_0,\ldots,x_{n-1}]$ (see \eqref{eq:tensor->map} for a precise statement). The eigenscheme of $A$ is again the subscheme of $\PP^{n-1}$ that corresponds to the fixed points of $\Psi$, i.e.~that is defined by the $2$-minors of \eqref{eq:2minors}. For a general tensor it is known that its eigenscheme is a zero-dimensional reduced subscheme of degree $\sum_{i=0}^{n-1}(d-1)^i$ (see \cite[Theorem 2.1]{ASS17}). For a more specific choice of tensor, the eigenscheme can be non-reduced or of positive dimension and hence it appears natural to study the discriminant of the eigenschemes of tensors of the same dimension and format. Such a discriminant has been recently introduced in \cite{ASS17} and called the \emph{eigendiscriminant}. It is a homogeneous polynomial in the coefficients of $d$-dimensional tensors of format $n\times\cdots\times n$ that vanishes for a given tensor $A$ if and only if its eigenscheme is non-reduced or of positive dimension. It is denoted by $\Delta_{n,d}$ and is of degree $n(n-1)(d-1)^{n-1}$ (see \cite[Corollary 4.2]{ASS17}).

From a geometric point of view, symmetric tensors of dimension $d$ and format $n\times n\times\cdots \times n$ are of particular interest. Indeed, they are in correspondence with homogeneous polynomials $\phi$ of degree $d$ in the variables $x_0,\ldots,x_{n-1}$. They are also in correspondence with the rational maps $\Psi$ defined by $\psi_i=\partial\phi/\partial x_i$, $i=0,\ldots,n-1$, for some polynomial $\phi$.  The eigenscheme of a symmetric tensor is hence the scheme of fixed points of the polar map of the corresponding polynomial $\phi$, in particular it contains the singular points of the hypersurface defined by the equation $\phi=0$.

\medskip

In this paper, we are interested in the computation of eigendiscriminants. Indeed, if the geometric properties of eigendiscriminants have already been explored (see \cite{ASS17,AboPreprint}), their computation in closed forms remains an open question. For instance, as noted in \cite[Example 4.5]{ASS17}, the computation of the eigendiscriminant $\Delta_{3,3}$ of a plane cubic, which is a degree 24 hypersurface in a $\PP^{9}$, is already a hard task (see also Example \ref{ex:ASS}). Even more difficult is the case of the eigendiscriminant $\Delta_{4,3}$ of a cubic surface in $\PP^3$. This is a polynomial of degree $96$ in a $\PP^{19}$. The problem of finding a closed formula for $\Delta_{4,3}$ is actually one of the ten open problems on cubic surfaces that are listed in  \cite{RaSt} (see \cite[Question 17]{RaSt}).

The main contributions of this paper are two closed formulas for the computation of the eigendiscriminants $\Delta_{3,d}$ and $\Delta_{4,d}$;  see Theorem \ref{thm:n=3} and Theorem \ref{thm:n=4}. Our approach can be seen as a generalization of the case $n=2$. In this case, the eigendiscriminant coincides with the classical discriminant of the minor $\delta:=x_0\psi_1(x_0,x_1)-x_1\psi_0(x_0,x_1)$:
$$\Disc(\delta)=\Delta_{2,d}(\Psi).$$
This property follows from the fact that the eigenscheme of a general tensor is a zero-dimensional complete intersection in this case. But when $n>2$, these eigenschemes are no longer complete intersections; they are determinantal schemes defined by the minors of \eqref{eq:2minors}. Thus, our strategy is to link these determinantal schemes to complete intersections that are defined by some $n-1$ minors of \eqref{eq:2minors}. The eigendiscriminant is then an irreducible component of such a discriminant and the main difficulty is to determine the other excess irreducible components in this complete intersection. This approach is motivated by the fact that the theory of discriminants of zero-dimensional complete intersections in a projective space is rich and offers an extensive formalism; see \cite{Krull1,Krull2,BJ14}. More broadly, we notice that computing some resultants or discriminants as factors of some other ones is a commonly used strategy in elimination theory, as for instance the computation of the hyperdeterminant of multi-dimensional matrices of certain formats, which is known as the Schl\"afli's method (see \cite{WZ96,Ott13}, and more generally \cite{GKZ}). 

\medskip

The paper is organized as follows. In Section \ref{sec:prelim} we briefly review some background material on resultants and discriminants that are used in this paper. Section \ref{sec:ternary} is devoted to the computation of the eigendiscriminant $\Delta_{3,d}$ of $d$-dimensional tensors of format $3\times 3 \times \cdots \times 3$. The main result, Theorem \ref{thm:n=3}, yields the irreducible decomposition of the discriminant of two minors of the matrix \eqref{eq:2minors} in this case. The consequences for the computation of  eigendiscriminants of plane cubics are also discussed. Finally, Section \ref{sec:quaternary} deals with the eigendiscriminant $\Delta_{4,d}$ of $d$-dimensional tensors of format $4\times 4 \times \cdots \times 4$. Similarly to Section \ref{sec:ternary}, the main result is the irreducible decomposition of the discriminant of three minors of the matrix \eqref{eq:2minors} (see Theorem \ref{thm:n=4}). In particular, this decomposition yields a closed formula for computing the eigendiscriminant of a cubic surface in $\PP^3$ as the ratio of some resultants and discriminants. 

\section{Notations and background material}\label{sec:prelim}

In order to fix our notation, we briefly review the definitions of resultants and of discriminants of zero-dimensional complete intersections in a projective space.   

\subsection{Resultants}

Let $d_0,\ldots,d_{n}$ be a sequence of positive integers and consider the generic homogeneous polynomials of degree $d_0,\ldots,d_{n}$   
in the variables $x_0,\ldots,x_n$:
$$f_i(x_0,\ldots,x_n)=\sum_{\alpha_0+\cdots +\alpha_n=d_i}u_{i,\alpha}x_0^{\alpha_0}x_1^{\alpha_1}\ldots x_n^{\alpha_n}, \ \ 
i=0,\ldots,n.$$
Let ${\UU}:=\ZZ[u_{i,\alpha}: i=0,\ldots,n, \alpha_0+\cdots +\alpha_n=d_i]$ be the universal ring of coefficients of the $f_i$'s. Thus, the polynomials $f_0,\ldots,f_n$ belong to the polynomial ring ${\UU}[x_0,\ldots,x_n]$ and one can show that the elimination ideal, also called the ideal of inertia forms, 
$$\left( (f_0,\ldots,f_n) : (x_0,\ldots,x_n)^\infty \right)\cap \UU$$
is a prime and principal ideal of $\UU$. The resultant is then defined as the unique generator, denoted $\Res_{d_0,\ldots,d_n}$ or simply $\Res$, of this ideal which satisfies the equality
\begin{equation*}
 \Res(x_0^{d_0},\ldots,x_n^{d_n})=1.
\end{equation*}
This notation means that the polynomial $\Res$ is equal to 1 when the coefficients $u_{i,\alpha}$ are specialized in such a way that each polynomial $f_i$ is specialized to $x_i^{d_i}$ for all $i$. 

The resultant of any given $(n+1)$-tuples of homogeneous polynomials is defined by specialization of $\Res$. More precisely, let $R$ be a commutative ring and $g_0,\ldots,g_n$ be homogeneous polynomials in $R[x_0,\ldots,x_n]$ of degree $d_0,\ldots,d_n$ respectively: 
    $$g_i=\sum_{\alpha_0+\cdots +\alpha_n=d_i}c_{i,\alpha}x_0^{\alpha_0}x_1^{\alpha_1}\ldots x_n^{\alpha_n}, \ \ 
i=0,\ldots,n.$$
The morphism of rings $\theta:{\UU}\rightarrow R:
    u_{j,\alpha} \mapsto c_{j,\alpha}$ corresponds to the
    specialization of the generic polynomials $f_0,\ldots,f_n$ to the polynomials
    $g_0,\ldots,g_n$ and the resultant of $g_0,\ldots,g_n$ is then defined by  
	$$\Res(g_0,\ldots,g_n):=\theta(\Res) \in R$$
(in particular observe that we have  $\Res=\Res(f_1,\ldots,f_n) \in \UU$).

The resultant of polynomials with coefficients in a field $\KK$ vanishes if and only if these polynomials have a common projective root over the algebraic closure of $\KK$. In addition, $\Res(f_0,\ldots,f_n)$ is homogeneous with respect to the coefficients of each polynomial $f_i$ of degree $d_0\ldots d_n/d_i$. Resultants have numerous properties  and there exists an extensive literature on this topic; we refer the reader to \cite{Jou91,GKZ,CLO} and the references therein. For instance, in what follows we will use several times the multiplicativity property of the resultant: if $g_0,g_0',g_1,\ldots,g_n$ are homogeneous polynomials with coefficients in any commutative ring then (see \cite[\S 5.7]{Jou91})
$$\Res(g_0g_0',g_1,\ldots,g_n)=\Res(g_0,g_1,\ldots,g_n)\Res(g_0',g_1,\ldots,g_n).$$
Another important point to notice is that several formulas are known to compute resultants (see for instance~\cite{Mac02}, \cite{Jou97},  \cite{GKZ}, \cite{CLO}, \cite{DaDi01}, and the references therein), the more classical one being the Macaulay formula that allows to compute the resultant as the ratio of two determinants; see also \cite{DaDi01} for a generalized and more compact version. 

\subsection{Discriminants of zero-dimensional complete intersections} Consider 
the generic homogeneous polynomials $f_0,\ldots,f_{n-1}$ of degree $d_0,\ldots,d_{n-1}$ in the variables $x_0,\ldots,x_n$ and denote by $\UU$ their universal ring of coefficients. We assume that $\sum_{i}(d_i-1)>0$. The Jacobian matrix of these polynomials has $n$ rows and $n+1$ columns. We denote by  $J_i(f_0,\ldots,f_{n-1})$ its signed $n$-minor which corresponds to removing the column number $i+1$, i.e.~the column that depends on the partial derivatives of the $f_j$'s with respect to the variable $x_i$. Then, the discriminant of $f_0,\ldots,f_{n-1}$, denoted $\Disc_{d_0,\ldots,d_{n-1}}$ or simply $\Disc$, is defined as the unique polynomial in $\UU$ which satisfies (one of) the equalities (see \cite[\S 3.1.2]{BJ14}):
\begin{equation}\label{eq:defDisc}
\Res(f_0,\ldots,f_{n-1},J_i)= \Disc\cdot \Res(f_0,\ldots,f_{n-1},x_i), \ i=0,\ldots,n.	
\end{equation}
From this definition, we deduce that for all $i$, $\Disc$ is homogeneous with respect to the coefficients of $f_i$ of degree 
\begin{equation}\label{eq:partialdegdisc}
\frac{d_0\ldots d_{n-1}}{d_i}\left((d_i-1) + \sum_{j=0}^{n-1}(d_j-1)\right).	
\end{equation}

As for resultants, the discriminant of homogeneous polynomials  $g_0,\ldots,g_{n-1}$ in $R[x_0,\ldots,x_n]$, $R$ being any commutative ring, is defined by specialization of the universal discriminant $\Disc$ and is denoted by $\Disc(g_0,\ldots,g_{n-1})$. This discriminant has many formal properties for which we refer the reader to \cite[Section 3]{BJ14} and the references therein. For instance, a property that we will use several times in this paper is the polarization formula: given homogeneous polynomials $g_0,g_0',g_1,\ldots,g_{n-1}$ of degree $d_0,d_0',d_1,\ldots,d_{n-1}$ we have
\begin{multline*}
\Disc(g_0g_0',g_1,\ldots,g_{n-1})=(-1)^{d_0d_0'd_1\ldots d_{n-1}}\Disc(g_0,g_1,\ldots,g_{n-1})\Disc(g_0',g_1,\ldots,g_{n-1}) \times \\ \Res(g_0,g_0',g_1,\ldots,g_{n-1})^2.	
\end{multline*}
We also emphasize that the definition of these discriminants via \eqref{eq:defDisc} gives explicit formulas to compute them by means of formulas  for computing resultants, as for instance the Macaulay formula.    

From a geometrical point of view, discriminants have the following interpretation: suppose we are given polynomials $g_0,\ldots,g_{n-1}$ with coefficients in an algebraically closed field $\KK$. These polynomials define $n$ hypersurfaces in $\PP^n_\KK$ and their discriminant vanishes if and only if they do not define a smooth zero-dimensional complete intersection in $\PP^n_\KK$. In other words, the universal discriminant $\Disc$ characterizes the collection of $n$ hypersurfaces of given degrees which do not form a smooth zero-dimensional complete intersection in $\PP^n_\KK$ \cite{BJ14,GKZ,Benoist}. 

\subsection{Eigendiscriminants}\label{subsec:eigdisc}
Consider a $d$-dimensional tensor $A={(a_{i_1,\ldots,i_d})}_{0\leq i_j\leq n-1}$ of format $n\times n\times\cdots \times n$. It defines the following rational map 
\begin{eqnarray*}
\Psi : \PP^{n-1} & \dasharrow & \PP^{n-1}\\
(x_0:\cdots:x_{n-1}) & \mapsto & (\psi_0:\ldots,\psi_{n-1})
\end{eqnarray*}
where the polynomials $\psi_0,\ldots,\psi_{n-1}$ are homogeneous polynomials of degree $d-1$:
\begin{equation}\label{eq:tensor->map}
\psi_i(x_0,\ldots,x_{n-1})= \sum_{j_2=0}^{n-1}\sum_{j_3=0}^{n-1}\cdots \sum_{j_d=0}^{n-1} a_{ij_2j_3\ldots j_d}x_{j_2}x_{j_3}\ldots x_{j_d}, \ i=0,\ldots,n-1.	
\end{equation}
As explained in Section \ref{sec:intro}, the eigendiscriminant of $A$ is the discriminant of the subscheme of $\PP^{n-1}$ defined by the fixed points of $\Psi$. By definition, the eigendiscriminant actually depends on the coefficients of $\Psi$ and therefore from now on we will refer to the eigendiscriminant of a rational map rather than a tensor. We also notice that in general there is not a unique tensor that yields a given rational map $\Psi$, but this is the case if one considers symmetric tensors.

Denote by $\UU$ the universal ring of coefficients of rational maps from $\PP^{n-1}$ to itself defined by homogeneous polynomials of degree $d-1$ (this is the universal ring of coefficients of the polynomials $\psi_0,\ldots,\psi_{n-1}$). As proved in \cite[Theorem 4.1]{ASS17}, the eigendiscriminant $\Delta_{n,d}$ is an irreducible homogeneous polynomial in $\UU$ of degree $n(n-1)(d-1)^{n-1}$. Similar to resultants and discriminants, the eigendiscriminant of any rational map $\Phi$ of the same format with coefficients in any commutative ring $R$ is defined by specialization of $\Delta_{n,d}$ and is denoted by $\Delta_{n,d}(\Phi) \in R$. We notice that $\Delta_{n,d}$ is only defined up to sign by the above property.

\medskip

In what follows, we will focus on the cases $n=3,4$ and provide universal formulas that relate eigendiscriminants, resultants and discriminants. Such formulas provide computational tools for the eigendiscriminants and could actually serve as alternative definitions. In particular, they allow their evaluation in many point (i.e.~rational maps) without ambiguity of sign or multiplicative constant, which is a very important property if one aims to compute classes of eigendiscriminants by means of interpolation methods.

\section{Eigendiscriminants of ternary forms}\label{sec:ternary}

In this section we focus on the eigendiscriminant $\Delta_{3,d}$ of rational maps $\Psi$ from $\PP^2$ to $\PP^2$ defined by three homogeneous polynomials of degree $d-1$; this setting corresponds to $d$-dimensional tensors $A$ of format $3\times 3\times \cdots \times 3$. 

Consider the generic  rational map
\begin{eqnarray*}
	\Psi : \PP^2 & \dasharrow & \PP^2 \\
	(x_0:x_1:x_2) & \mapsto & (\psi_0:\psi_1:\psi_2)
\end{eqnarray*}
where $\psi_0,\psi_1,\psi_2$ are generic homogeneous polynomials of degree $d-1$. We assume $d\geq 3$ and denote by $\UU$ the universal ring of coefficients of the $\psi_i$'s. The eigenscheme of $A$ is the subscheme in $\PP^2_{\UU}$ defined by the 2-minors of the matrix
\begin{equation}\label{eq:psi-matrix}
\left(
\begin{array}{ccc}
	\psi_0 & \psi_1 & \psi_2 \\
	x_0 & x_1 & x_2
\end{array}
\right).	
\end{equation}
We denote them by 
$$\delta_0=
\det \left( 
\begin{array}{cc}
	\psi_1 & \psi_2  \\
	x_1 & x_2 
\end{array}
\right), \ 
\delta_1=
\det \left( 
\begin{array}{cc}
	\psi_0 & \psi_2  \\
	x_0 & x_2 
\end{array}
\right), \ 
\delta_2=
\det \left( 
\begin{array}{cc}
	\psi_0 & \psi_1  \\
	x_0 & x_1
\end{array}
\right).
$$

The eigenscheme of $\Psi$ is clearly contained in the complete intersection schemes defined by any two of these minors, say $\delta_0$ and $\delta_1$. It follows that $\Delta_{3,d}(\Psi)$ is an irreducible factor of the discriminant $\Disc(\delta_0,\delta_1)$. The next result provides the complete decomposition into irreducible factors of this discriminant. 
% It can be used to compute $\Delta_{3,d}$ by means of resultants and discriminants computations (see Section \ref{sec:prelim}).

Given a homogeneous polynomial $p$ in the variables $x_0,x_1,x_2$ and an integer $k\in \{0,1,2\}$, the notation $\overline{p}^{(k)}$ denotes the homogeneous polynomial $p(x_k=0)$ in two variables obtained by specializing $x_k$ to 0 in $p$.

\begin{thm}\label{thm:n=3} Let $i,j,k$ be integers such that $\{i,j,k\}=\{0,1,2\}$. Then, we have the following equality in the universal ring $\UU$ of coefficients of $\Psi$: 
	$$\Disc(\delta_i,\delta_j)=\Delta_{3,d}(\Psi) \, \Disc\left(\overline{\psi_k}^{(k)}\right) \, \Res\left(\overline{\delta_k}^{(k)},\overline{\psi_k}^{(k)} \right)^2.$$	
\end{thm}
\begin{proof} For the sake of simplicity, we prove this formula in the case $i=0$, $j=1$, $k=2$, the other cases being obtained in a similar way. 
	
	First, we consider the discriminant of $x_0\delta_0$ and $\delta_1$ and we decompose it as follows (see \cite[Section 3]{BJ14} for the properties of discriminants):
	\begin{align*}
	\Disc(x_0\delta_0,\delta_1) &= (-1)^{d^2}\Disc(x_0,\delta_1)\Disc(\delta_0,\delta_1)\Res(x_0,\delta_0,\delta_1)^2\\
	&= (-1)^{d^2+d}\Disc(-x_2\psi_0(0,x_1,x_2)) \Disc(\delta_0,\delta_1) \\  
	&  \hspace{2cm} \times \Res(x_2\psi_1(0,x_1,x_2)-x_1\psi_2(0,x_1,x_2),-x_2\psi_0(0,x_1,x_2))^2\\
	&= (-1)^{d-1}\Disc(\psi_0(0,x_1,x_2)) \, \psi_0(0,1,0)^2 \, \Disc(\delta_0,\delta_1) \\  
	&  \hspace{2cm} \times \psi_2(0,1,0)^2 \, \Res(x_2\psi_1(0,x_1,x_2)-x_1\psi_2(0,x_1,x_2),\psi_0(0,x_1,x_2))^2.
	\end{align*}
	On the other hand, since $x_0\delta_0+x_1 \delta_1+x_2\delta_2=0$, we also have
	\begin{align*}
	 \Disc(x_0\delta_0,\delta_1) &= \Disc(-x_1\delta_1-x_2\delta_2,\delta_1)=\Disc(-x_2\delta_2,\delta_1)= \Disc(x_2\delta_2,\delta_1)\\
	  &= (-1)^{d^2}\Disc(x_2,\delta_1)\Disc(\delta_2,\delta_1) \Res(x_2,\delta_2,\delta_1)^2\\
	  &= (-1)^{d^2+d} \Disc(x_0\psi_2(x_0,x_1,0))\Disc(\delta_1,\delta_2)  \\
	  & \hspace{2cm} \times \psi_0(0,1,0)^2 \, \Res(x_1\psi_0(x_0,x_1,0)-x_0\psi_1(x_0,x_1,0),\psi_2(x_0,x_1,0))^2 \\
	  &=  (-1)^{d-1} \psi_2(0,1,0)^2 \, \Disc(\psi_2(x_0,x_1,0)) \Disc(\delta_1,\delta_2) \\
 & \hspace{2cm} \times \psi_0(0,1,0)^2 \, \Res(x_1\psi_0(x_0,x_1,0)-x_0\psi_1(x_0,x_1,0),\psi_2(x_0,x_1,0))^2.
	\end{align*}
Therefore, from these two computations we obtain the following equality in $\UU$ (which is a UFD):
\begin{multline}\label{eq:n=3-D01D12}
\Disc\left(\overline{\psi_0}^{(0)}\right)\Disc(\delta_0,\delta_1)
\Res\left(x_2\overline{\psi_1}^{(0)}-x_1\overline{\psi_2}^{(0)},\overline{\psi_0}^{(0)}\right)^2 \\ = \Disc\left(\overline{\psi_2}^{(2)}\right)\Disc(\delta_1,\delta_2)
\Res\left(x_1\overline{\psi_0}^{(2)}-x_0\overline{\psi_1}^{(2)},\overline{\psi_2}^{(2)}\right)^2.
\end{multline}

Inspecting \eqref{eq:n=3-D01D12}, we first notice that for all $k$ the discriminant $\Disc(\overline{\psi_k}^{(k)})$ is an irreducible polynomial because $\overline{\psi_k}^{(k)}$ is a generic polynomial. In addition, these discriminants are also distinct because they do not depend on the same variables. The resultant
\begin{equation}\label{eq:ResIrred}
	\Res\left(x_1\overline{\psi_0}^{(2)}-x_0\overline{\psi_1}^{(2)},\overline{\psi_2}^{(2)}\right)
\end{equation}
is also irreducible because one can choose a specialization of $\psi_0$ and $\psi_1$ such that the polynomial  $x_1\overline{\psi_0}^{(2)}-x_0\overline{\psi_1}^{(2)}$ in two variables $x_0,x_1$ is generic. 
The same irreducibility property holds for the other resultant appearing in equation \eqref{eq:n=3-D01D12} and hence, from all these considerations we deduce that the product of irreducible and coprime polynomials 
\begin{equation}\label{eq:pd-discres2}
\Disc\left(\overline{\psi_2}^{(2)}\right)\Res\left(x_1\overline{\psi_0}^{(2)}-x_0\overline{\psi_1}^{(2)},\overline{\psi_2}^{(2)}\right)^2
\end{equation}  
divides the discriminant $\Disc(\delta_0,\delta_1)$ in $\UU$. In addition, the eigendiscriminant $\Delta_{3,d}$ also divides $\Disc(\delta_0,\delta_1)$ by definition, and since it is irreducible we deduce that it is another irreducible factor of $\Disc(\delta_0,\delta_1)$.

Now, $\Disc(\delta_0,\delta_1)$ is a homogeneous polynomial in $\UU$ of degree $6(d-1)d$. The discriminant 
$\Disc\left(\overline{\psi_2}^{(2)}\right)$ is homogeneous in the coefficients of $\psi_2$, hence of $\Psi$, of degree $2(d-2).$ The resultant 
$$\Res\left(x_1\overline{\psi_0}^{(2)}-x_0\overline{\psi_1}^{(2)},\overline{\psi_2}^{(2)} \right)$$ 
is homogeneous in the coefficients of $\psi_2$ of degree $d$ and in the coefficients of $\psi_0$ and $\psi_1$ of degree $d-1$, hence, it is homogeneous of degree $2d-1$ in the coefficients of $\Psi$. Finally, the eigendiscriminant $\Delta_{3,d}(\Psi)$ is of degree $6(d-1)^2$. Therefore, by comparison of the degrees of these quantities we deduce that there exists a nonzero integer $c$ such that
\begin{equation}\label{eq:firstdecomp-n=3}
\Disc(\delta_0,\delta_1)=c\, \Delta_{3,d}(\Psi)\, \Disc\left(\overline{\psi_2}^{(2)}\right)\Res\left(x_1\overline{\psi_0}^{(2)}-x_0\overline{\psi_1}^{(2)},\overline{\psi_2}^{(2)}\right)^2.	
\end{equation}

To conclude the proof, it remains to show that $c$ is invertible in $\ZZ$. We first notice that the polynomials on the right-hand side of \eqref{eq:firstdecomp-n=3} are all primitive polynomials. Now, consider the specialization that sends  $\psi_1$ to 0 and leaves the other coefficients invariant. The discriminant $\Disc(\delta_0,\delta_1)$ is then specialized to
\begin{align*}
	\Disc(-x_1\psi_2,x_2\psi_0-x_0\psi_2) &= \Disc(x_1,x_2\psi_0-x_0\psi_2)\Disc(\psi_2,x_2\psi_0-x_0\psi_2) \times \\
	& \hspace{6cm}  \Res(x_1,\psi_2,x_2\psi_0-x_0\psi_2)^2 \\
	&=  \Disc(x_2\overline{\psi_0}^{(1)}-x_0\overline{\psi_2}^{(1)})\Disc(\overline{\psi_2}^{(2)})\Disc(\psi_2,\psi_0)\times \\
		& \hspace{3cm}  \psi_2(1,0,0)^2 \Res(\overline{\psi_0}^{(2)},\overline{\psi_2}^{(2)})^2  \Res(\overline{\psi_2}^{(1)},\overline{\psi_0}^{(1)})^2.
\end{align*}
In this decomposition, all polynomials on the right-hand side are primitive, including the discriminant of $x_2\overline{\psi_0}^{(1)}-x_0\overline{\psi_2}^{(1)}$ (because one can specialize it to a generic polynomial in two variables). Therefore, we conclude that $c=\pm 1$. Since the sign of $\Delta_{3,d}(\Psi)$ is not yet set, one can assume that $c=1$ and hence use the decomposition we just proved to normalize this eigendiscriminant.
\end{proof}

Theorem \ref{thm:n=3} yields a universal formula (i.e.~it remains valid under any specialization) from which we can derive three consequences. First, this formula provides an alternative definition of the eigendiscriminant based on the definitions of resultants and discriminants. In particular, the sign of $\Delta_{3,d}$ can be set by relying on the normalizations of resultants and discriminants (see Section \ref{sec:prelim}), so that the polynomial $\Delta_{3,d}$ is defined without ambiguity. This is a very important property in order to evaluate correctly $\Delta_{3,d}$ for any given rational map.

As a second consequence of Theorem \ref{thm:n=3}, properties of $\Delta_{3,d}$ can be derived from the known properties of resultants and discriminants. For instance, we get that $\Delta_{3,d}$ must be a square over a field of characteristic 2 (see \cite[Theorem 3.24]{BJ14}). As another illustration, we mention the following precise invariance property of $\Delta_{3,d}$ under the canonical action of $\mathrm{GL}_3$.

\begin{cor} Let $R$ be a commutative ring, $\Phi$ a rational map from $\PP^2_R$ to itself defined by polynomials of degree $d-1$ and $\varphi=[c_{i,j}]_{0\leq i,j\leq 2}$ be a matrix with entries in $R$. Then,  
	$$ \Delta_{3,d}(\Phi\circ \varphi)=\det(\varphi)^{2(d-1)(d^2-d+1)} \Delta_{3,d}(\Phi)$$
where the composition $\Phi\circ \varphi$ stands for the linear change of coordinates in $\PP^2_R$ defined by $\varphi$. 
\end{cor}
\begin{proof} Using Theorem \ref{thm:n=3}, the claimed formula follows straightforwardly from the corresponding properties of resultants \cite[\S 5.13]{Jou91} and discriminants \cite[Proposition 3.27]{BJ14}.
\end{proof}

The third consequence of Theorem \ref{thm:n=3} is the computation of $\Delta_{3,d}$ as a ratio of determinants. To be more precise, one has to compute the discriminant $\Disc(\delta_0,\delta_1)$ and two other terms that correspond to determinants of Sylvester matrices. The discriminant $\Disc(\delta_0,\delta_1)$ can be computed as follows. By definition we have
\begin{equation}\label{eq:DiscComp}
\Res(\delta_0,\delta_1,x_1)\cdot\Disc(\delta_0,\delta_1)= \Res(\delta_0,\delta_1,J_{1}(\delta_0,\delta_1))	
\end{equation}
where the Jacobian determinant $J_1$ is defined by
$$J_{1}(\delta_0,\delta_1):=
\left|
\begin{array}{cc}
	\partial_0 \delta_0 & \partial_2 \delta_0 \\
	\partial_0 \delta_1 & \partial_2 \delta_1 
\end{array}
\right|=
\left|
\begin{array}{cc}
	x_2\partial_0\psi_1-x_1\partial_0\psi_2 & x_2\partial_2\psi_1-x_1\partial_2\psi_2 +\psi_1 \\
	x_0\partial_0\psi_2-x_2\partial_0\psi_0 +\psi_2 & x_0\partial_2\psi_2-x_2\partial_2\psi_0-\psi_0  
\end{array}
\right|.$$
The resultant in the left-hand side of \eqref{eq:DiscComp} is equal to
$$\Res(\delta_0,\delta_1,x_1)=\Res\left(x_2\overline{\psi_1}^{(1)},x_0\overline{\psi_2}^{(1)}-x_2\overline{\psi_0}^{(1)}\right)=\psi_2(1,0,0)\cdot \Res\left(\overline{\psi_1}^{(1)},x_0\overline{\psi_2}^{(1)}-x_2\overline{\psi_0}^{(1)}\right),$$
so its computation is reduced to a Sylvester determinant. It turns out that the resultant in the right-hand side of \eqref{eq:DiscComp}  can be computed as the determinant of a mixed resultant matrix involving Macaulay and Bezoutian blocks; see \cite[3.11.19.25]{Jou97} for more details. 

In summary, Theorem \ref{thm:n=3} shows that $\Delta_{3,d}(\Psi)$ can be computed as the ratio of a determinant of degree $2d(3d-2)$ divided by the product of the coefficient $\psi_2(1,0,0)$ and three Sylvester determinants of degree $2d-1$, $2d-1$ and $2(d-1)$ respectively. We notice that for some particular $\Psi_0$, the formula given in Theorem \ref{thm:n=3} may vanish identically. In such cases, a more general family of maps of the same format must be considered, for which the eigendiscriminant can be computed by means of Theorem \ref{thm:n=3}. Then, $\Delta_{3,d}(\Psi_0)$ is obtained by specialization. This approach is illustrated in Example \ref{ex:ASS} (second case). 

\medskip

Before closing this section, we mention the particular case of polar maps that are in correspondence with symmetric tensors. In our setting, the map $\Psi$ corresponds to the polar map of a plane curve of equation $\phi(x_0,x_1,x_2)=0$, where $\phi$ is a homogeneous polynomial of degree $d$; the polynomials $\psi_i$ are then the partial derivatives of $\phi$. As observed in \cite[Example 4.4]{ASS17}, the computation of the eigendiscriminant of a plane cubic curve is already a hard task; it is a homogeneous polynomial of degree 24 in 10 variables. The decomposition formula we obtained in Theorem \ref{thm:n=3} provides a way to compute with this eigendiscriminant, a property we illustrate by revisiting two examples from \cite{ASS17}.

\begin{ex}\label{ex:ASS} Consider the polynomial
	$$ C:= \left( 2\,x_0+x_1 \right)  \left( 2\,x_0+x_2 \right)  \left( 2\,x_1+x_2 \right) +
ux_0x_1x_2.$$
Its eigendiscriminant is equal to the product
{\scriptsize
$$
2^4\left(-16\,{u}^{24}-2304\,{u}^{23}-152784\,{u}^{22}-6097536\,{u}^{21}-
159761808\,{u}^{20}-2779161840\,{u}^{19}-29727588168\,{u}^{18}-\right.$$
$$124641852624\,{u}^{17}+1234078589016\,{u}^{16}+18314627517360\,{u}^{15
}+8929524942432\,{u}^{14}-1200933047925648\,{u}^{13}-$$
$$3722203539791685
\,{u}^{12}+63418425922462464\,{u}^{11}+257381788882972176\,{u}^{10}-
2676970903961440800\,{u}^{9}-$$
$$7927655114836286496\,{u}^{8}+89013482239908955392\,{u}^{7}+13934355026171012352\,{u}^{6}-
1729250356371556792320\,{u}^{5}+$$
$$5159222324901192930048\,{u}^{4}+ 11838757458480721920\,{u}^{3}-28255456641734116982784\,{u}^{2}+$$
$$\left.56809371779894977339392\,u-37304830510913780269056\right)$$}	

\vspace{-1em}
\noindent where the second factor is a primitive polynomial that appeared in \cite[Example 4.4]{ASS17}. The factor $2^4$ shows that in characteristic 2 the eigendiscriminant of $C$ must be equal to 0; it is obtained thanks to Theorem \ref{thm:n=3} (compare with \cite[Example 4.4]{ASS17}). 
	
Now, consider the polynomial 
$$C:= ux_0^3+vx_1^3+wx_2^3+x_0x_1x_2.$$	
In order to compute its eigendiscriminant with the decomposition formula of Theorem \ref{thm:n=3} it is necessary to see it as a particular case of a family of curves. We consider the family of curves
$$C_t:=ux_0^3+vx_1^3+wx_2^3+x_0x_1x_2+t(x_1^2x_2+x_0^2x_1+x_0^2x_2+x_0x_1^2+x_1x_2^2).$$
The eigendiscriminant can then be computed and after evaluation at $t=0$ we get the eigendiscriminant of $C$ which is equal to the following product:
{\scriptsize
$$
 \left( 3\,w-1 \right) ^{2} \left( 3\,w+1 \right) ^{2} \left( 3\,v-1
 \right) ^{2} \left( 3\,v+1 \right) ^{2} \left( 3\,u-1 \right) ^{2}
 \left( 3\,u+1 \right) ^{2}\times $$
 $$ \left( -531441\,{u}^{4}{v}^{4}{w}^{4}+
708588\,{u}^{5}{v}^{3}{w}^{3}+708588\,{u}^{3}{v}^{5}{w}^{3}+708588\,{u
}^{3}{v}^{3}{w}^{5}-1062882\,{u}^{4}{v}^{4}{w}^{2}-1062882\,{u}^{4}{v}
^{2}{w}^{4}-\right.$$
$$1062882\,{u}^{2}{v}^{4}{w}^{4}+1810836\,{u}^{3}{v}^{3}{w}^
{3}+177147\,{u}^{4}{v}^{4}-39366\,{u}^{4}{v}^{2}{w}^{2}+177147\,{u}^{4
}{w}^{4}-39366\,{u}^{2}{v}^{4}{w}^{2}-39366\,{u}^{2}{v}^{2}{w}^{4}+$$
$$177147\,{v}^{4}{w}^{4}-314928\,{u}^{3}{v}^{3}w-314928\,{u}^{3}v{w}^{3}
-314928\,u{v}^{3}{w}^{3}+244944\,{u}^{2}{v}^{2}{w}^{2}+46656\,{u}^{3}v
w+46656\,u{v}^{3}w+46656\,uv{w}^{3}-$$
$$\left.23328\,{u}^{2}{v}^{2}-23328\,{u}^{
2}{w}^{2}-23328\,{v}^{2}{w}^{2}-6912\,uvw+2304\,{u}^{2}+2304\,{v}^{2}+
2304\,{w}^{2}-256 \right) 
$$
}

\vspace{-1em}
\noindent We notice that only the last factor of the above product was wrongly identified as the eigendiscriminant of $C$ in \cite[Example 4.4]{ASS17}.
\end{ex}

\section{Eigendiscriminants of quaternary forms}\label{sec:quaternary}

In this section we focus on the eigendiscriminant $\Delta_{4,d}$ of rational maps $\Psi$ from $\PP^3$ to $\PP^3$ defined by three homogeneous polynomials of degree $d-1$; this setting corresponds to $d$-dimensional tensors $A$ of format $4\times 4\times \cdots \times 4$. 

Consider the generic  rational map
\begin{eqnarray*}
	\Psi : \PP^3 & \dasharrow & \PP^3 \\
	(x_0:x_1:x_2:x_3) & \mapsto & (\psi_0:\psi_1:\psi_2:\psi_3)
\end{eqnarray*}
where $\psi_0,\psi_1,\psi_2,\psi_3$ are homogeneous polynomials of degree $d-1$, $d\geq 3$. We denote by $\UU$ the universal ring (over the integers) of coefficients of the $\psi_i$'s. The eigenscheme of $\Psi$ is the subscheme of $\PP^3_\UU$ defined by the $2$-minors of the matrix
\begin{equation}\label{eq:psi3-matrix}
\left(
\begin{array}{cccc}
	\psi_0 & \psi_1 & \psi_2 & \psi_3 \\
	x_0 & x_1 & x_2 & x_3
\end{array}
\right).	
\end{equation}
We denote these minors by $\delta_{ij}:=x_i\psi_j-x_j\psi_i$, where $i,j$ are distinct integers in the set $\{0,1,2,3\}$. They are homogeneous polynomials of degree $d$.

In what follows, we will use the following notation: given an homogeneous polynomial $p$ in the variables $x_0,x_1,x_2,x_3$ and an integer $i\in \{0,1,2,3\}$,  $\overline{p}^{(i)}$ denotes the homogeneous polynomial $p(x_i=0)$ in three variables obtained by specializing $x_i$ to 0 in $p$. Similarly,  $\overline{p}^{(i,j)}$, $i,j$ being  two distinct integers in $\{0,1,2,3\}$, denotes the homogeneous polynomial $p(x_i=0,x_j=0)$ in two variables obtained by specializing $x_i$ and $x_j$ to 0 in $p$.

\medskip

We begin with two preliminary technical results.

\begin{lem}\label{lem:resIrr} Let $p,p',q,r,s$ be generic homogeneous polynomials in the variables $x_0,x_1,x_2$ such that $\deg(p),\deg(p')\geq 1$ and $\deg(q)=\deg(r)=\deg(s) \geq 0$. Then, the resultant
	$$\Res(p,p',x_0q-x_1r)$$
is an irreducible polynomial in the universal ring of coefficients of $p,p',q,r$. In addition, there exists an irreducible polynomial $\Rc$ such that 
$$\Res(p,x_0q-x_1r,x_1s-x_2q)=\Res(\overline{p}^{(1)},\overline{q}^{(1)})\Rc$$
in the universal ring of coefficients of $p,q,r,s$.
\end{lem}
\begin{proof} We begin with the first resultant $\Ec:=\Res(p,p',x_0q-x_1r)$. We denote by $d_1$ and $d_2$ the degrees of $p$ and $p'$ respectively. To prove the claimed property we  proceed by induction on the sum $d_1+d_2$. So, we begin with the case $d_1=d_2=1$; we set $p=a_0x_0+a_1x_1+a_2x_2$, $p'=b_0x_0+b_1x_1+b_2x_2$ and 
	$$m_0=\det \left(              
	\begin{array}{cc}
		a_1 & b_1 \\
		a_2 & b_2
	\end{array}\right),
	m_1=-\det \left(              
	\begin{array}{cc}
		a_0 & b_0 \\
		a_2 & b_2
	\end{array}\right),
	m_2=\det \left(              
	\begin{array}{cc}
		a_0 & b_0 \\
		a_1 & b_1
	\end{array}\right).
	$$
	By a classical property of the resultant \cite[Proposition 5.4.4]{Jou91} we have
	$$\Res(p,p',x_0q-x_1r)= 
	m_0q(m_0,m_1,m_2)-m_1r(m_0,m_1,m_2).$$
	As $q$ and $r$ are generic polynomials (and also $p$ and $p'$), this resultant is irreducible. Now, let $d\geq 3$ be any integer and assume that the resultant $\Ec$ is irreducible for all generic polynomials $p$, $p'$ of degree $d_1,d_2$ such that $d_1+d_2\leq d-1$. One can assume that $d_1\geq 2$, because otherwise one can exchange the roles of $p$ and $p'$. Consider the specialization that sends $p$ to the product of two generic polynomials $p_1$ and $p_2$ of positive degrees. Then, $\Ec$ specializes to the product 
	$$\Res(p_1,p',x_0q-x_1r)\Res(p_2,p',x_0q-x_1r).$$
	But by our inductive assumption, these two resultants are irreducible. Moreover, in view of this decomposition any irreducible factor of $\Ec$ must depend on $p$, and therefore depends on both $p_1$ and $p_2$. We conclude that $\Ec$ is irreducible, as claimed. 
	
\medskip

	We turn to the second claim of this lemma on the irreducible decomposition of the resultant $\Ec:=\Res(p,x_0q-x_1r,x_1s-x_2q)$. We first observe that the inclusion of ideals $$(p,x_0q-x_1r,x_1s-x_2q)\subset (p,q,x_1)$$ implies, by the divisibility property of resultants \cite[\S 5.6]{Jou91}, the existence of $\Rc$ such that
 	 $$\Ec=\Res(\overline{p}^{(1)},\overline{q}^{(1)})\Rc.$$
This equality holds in the universal ring of coefficients of $p,q,r,s$ that we denote by $\AA$. We also notice that $\Res(\overline{p}^{(1)},\overline{q}^{(1)})$ is irreducible as $\overline{p}^{(1)}$ and $\overline{q}^{(1)}$ are generic polynomials in the variables $x_0,x_2$.

Let $\AA_{\overline{\QQ}}:=\AA\otimes_\ZZ \overline{\QQ}$ be the extension of the ring $\AA$ over the algebraic closure $\overline{\QQ}$ of the field of rational numbers and consider the incidence variety $Z=V(p,x_0q-x_1r,x_1s-x_2q)\subset \PP^2_{\overline{\QQ}}\times \Spec(\AA_{\overline{\QQ}})$, as well as the two canonical projections 
$$ \pi_1: \PP^2_{\overline{\QQ}}\times \Spec(\AA_{\overline{\QQ}}) \rightarrow \PP^2_{\overline{\QQ}}, \ \pi_2: \PP^2_{\overline{\QQ}}\times \Spec(\AA_{\overline{\QQ}}) \rightarrow \Spec(\AA_{\overline{\QQ}}).$$ 
Since $\pi_2(Z)=V(\Ec)$, we analyze the irreducible components of $Z$. Let $U$ the open subset of $\PP^2$ such that $x_1\neq 0$. For any point $x\in U$, $\pi_1^{-1}(x)$ is a linear space of codimension 3 in $\Spec(\AA_{\overline{\QQ}})$ because $p,r$ and $s$ are generic polynomials. It follows that the restriction $Z_{|U}$ of $Z$ to $U\times \Spec(\AA_{\overline{\QQ}})$ is irreducible. 
Moreover, for any point $x \notin U$, i.e.~$x\in V(x_1)$, the fiber $\pi_1^{-1}(x)$ is defined by the three equations $\overline{p}^{(1)}=0,x_0\overline{q}^{(1)}=0,x_2\overline{q}^{(1)}=0$, or equivalently by the two equations $\overline{p}^{(1)}=0$ and $\overline{q}^{(1)}=0$ (by saturation with respect to the ideal $(x_0,x_2)$). 
So the restriction $Z_{|V(x_1)}$ of $Z$ to $U\times \Spec(\AA_{\overline{\QQ}})$ is equal to $V(\overline{p}^{(1)},\overline{q}^{(1)})$, which is irreducible. From these considerations, we deduce that $\pi_2(Z)=V(\Ec)$ is composed of two irreducible components: the projection of $Z_{|V(x_1)}$ by $\pi_2$, which is nothing but $V(\Res(\overline{p}^{(1)},\overline{q}^{(1)}))$, and 
the projection of the algebraic closure of $Z_{|U}$ by $\pi_2$, which corresponds to $V(\Rc)$. 
Since $\Ec$ is defined over the integers, we deduce that there exists a nonzero rational number $c\in \QQ$ such that 
$$\Ec=c\, \Res(\overline{p}^{(1)},\overline{q}^{(1)}))\Rc.$$

To conclude, it remains to show that $c=\pm 1$. For that purpose, we come back to the universal ring $\AA$ and consider the specialization which sends $r$ to 0 and leaves invariant $p,q,s$. The resultant $\Ec$ is then specialized to 
\begin{align*}
	\Res(p,x_0q-x_1r,x_1s-x_2q) &= \Res(p,x_0q,x_1s-x_2q) \\
			&= \Res(p,x_0,x_1s-x_2q)\Res(p,q,x_1s-x_2q) \\
			&= \Res(\overline{p}^{(0)},x_1\overline{q}^{(0)}-x_2\overline{s}^{(0)})\Res(p,q,x_1s)\\
			&= \Res(\overline{p}^{(0)},x_1\overline{q}^{(0)}-x_2\overline{s}^{(0)})\Res(\overline{p}^{(1)},\overline{q}^{(1)}) \Res(p,q,s).
\end{align*}
This is a product of three irreducible polynomials in $\AA$: this is clear for the two last ones, and for the first one we simply notice that the polynomials $q$ and $s$ can be specialized such that $x_1\overline{q}^{(0)}-x_2\overline{s}^{(0)}$ is a generic polynomial in the two variables $x_1,x_2$. From here we deduce that $c=\pm 1$, which concludes the proof.
\end{proof}

\begin{lem}\label{lem:gendisc} Let $p,q,r$ be three generic homogeneous polynomials in the variables $x_0,x_1,x_2$ such that $deg(p)\geq 1$ and $\deg(q)=\deg(r)\geq 0$. Then, the discriminant $\Disc(p,x_0q-x_1r)$ is an irreducible polynomial in the universal ring of coefficients of $p,q$ and $r$.
\end{lem}
\begin{proof} Let us denote by $d$ and $e$ the degree of $p$ and $x_0q-x_1r$ respectively, and by $\AA$ the universal ring of coefficients of $p$, $q$ and $r$. We also set $\Dc:=\Disc(p,x_0q-x_1r) \in \AA$.
	
	We  proceed by induction on $d$. We begin with the case $d=1$ where $p$ is the generic linear form $a_0x_0+a_1x_1+a_2x_2$. Consider the specialization that sends $p$ to $a_2x_2$; $\Dc$ is specialized to the product
	$$ \Disc(a_2x_2,x_0q-x_1r)=a_2^{d(d-1)}\Disc(x_2,x_0q-x_1r)=a_2^{d(d-1)}\Disc(x_0\overline{q}^{(2)}-x_1\overline{r}^{(2)}).$$
The discriminant $\Disc(x_0\overline{q}^{(2)}-x_1\overline{r}^{(2)})$ is irreducible because $q$ and $r$ can be specialized such that $x_0\overline{q}^{(2)}-x_1\overline{r}^{(2)}$ is the generic homogeneous polynomial of degree $e$ in the two variables $x_0,x_1$. We deduce that $\Dc=P(p)Q(p,q,r)$ where $Q(p,q,r)$ is an irreducible polynomial that depends on $p,q,r$ and $P(p)$ is a polynomial that only depends on $p$, because any irreducible factor of $\Dc$ which  depends on $q$ and $r$ must yield at least one irreducible factor that depends on $q$ and $r$ via specialization. Now, to determine $P(p)$ we consider the specialization that sends $r$ to 0 and leaves $p$ and $q$ invariant. $\Dc$ is specialized to
$$ P(p)Q(p,q,0)=\Disc(p,x_0q)=\Disc(\overline{p}^{(0)})\Disc(p,q)\Res(\overline{p}^{(0)},\overline{q}^{(0)})^2.$$
As all the factors on the right-hand side of this equality are irreducible polynomials (as discriminants and resultants  of generic polynomials), we deduce that $P(p)$ is either equal to $\Disc(\overline{p}^{(0)})$ or is invertible, i.e.~is equal to $\pm 1$. Now, consider the specialization that sends $q$ to 0 and leaves $p$ and $r$ invariant. The same reasoning shows that $P(p)$ is either equal to $\Disc(\overline{p}^{(1)})$ or is invertible, i.e.~is equal to $\pm 1$. Since $\Disc(\overline{p}^{(0)})$ and $\Disc(\overline{p}^{(1)})$ are two distinct irreducible polynomials we deduce that $P(p)=\pm 1$, and consequently that $\Dc$ is irreducible if $d=1$.

\medskip

To proceed by induction, let $d$ be any integer and let us assume that $\Dc$ is irreducible for a polynomial $p$ of any degree $\leq d-1$. Let also $k$ be any integer such that $1\leq k \leq d-1$ and consider the specialization $\rho$ that sends $p$ to the product $p_1p_2$ of two generic polynomials of degree $k$ and $d-k$ respectively. Then, 
\begin{equation}\label{eq:polDiscpqr}
	\rho(\Dc)=\Disc(p_1,x_0q-x_1r)\Disc(p_2,x_0q-x_1r)\Res(p_1,p_2,x_0q-x_1r)^2.
\end{equation} 
Let us denote by $\Dc_1,\Dc_2$ and $\Rc$, respectively, the three factors in this product. The two discriminants $\Dc_1$ and $\Dc_2$ are irreducible by our induction hypothesis. The resultant factor $\Rc$ is also irreducible by Lemma \ref{lem:resIrr}, so we deduce that $\Dc$ has at most three irreducible factors and that they all depend on $p,q$ and $r$. Moreover, since the specialization of each of these irreducible factors must depend on $p_1$ and $p_2$ we deduce that the two discriminants in \eqref{eq:polDiscpqr} necessarily comes from the same irreducible factor of $\Dc$. Therefore, there are 4 possible configurations: (1) $\Dc$ is irreducible, (2) $\Dc$ is the product of two irreducible polynomials $P_1$ and $P_2$ such that 
$$\rho(P_1)=\Dc_1\Dc_2, \ \rho(P_2)=\Rc^2,$$
(3) $\Dc$ is the product of two irreducible polynomials $P_1$ and $P_2$ such that 
$$\rho(P_1)=\Dc_1\Dc_2\Rc, \ \rho(P_2)=\Rc,$$ 
or (4)   $\Dc$ is the product of three irreducible polynomials $P_1$, $P_2$ and $P_3$ such that 
$$
\rho(P_1)=\Dc_1\Dc_2, \ \rho(P_2)=\Rc, \ \rho(P_3)=\Rc.$$
But inspecting the degree of $\Dc$, $\Dc_1\Dc_2$ and $\Rc$ with respect to the coefficients $q$ and $r$ we get
$$ \deg_{q,r}\Dc= d(d-3+2e), \ \deg_{q,r}(\Dc_1\Dc_2)={d}^{2}+(2e-2k-3) d+2{k}^{2}, \ \deg_{q,r}(\Rc)=k(d-k),$$
and hence, taking into account the fact that the integer $k$ can vary from $1$ to $d-1$, we deduce that configurations (2), (3) and (4) are not possible. Therefore, $\Dc$ is an irreducible polynomial.
\end{proof}

We are now ready to state our main result. 

\begin{thm}\label{thm:n=4} Let $i,j,k,l$ be integers such that $\{i,j,k,l\}=\{0,1,2,3\}$. Then, we have the following irreducible decomposition in $\UU$:
\begin{multline*}
\Disc(\delta_{ki},\delta_{ij},\delta_{jl}) = \Delta_{4,d}(\Psi) 
\Disc(\overline{\psi_i}^{(i)},\overline{\delta_{jl}}^{(i)}) 
\Disc(\overline{\psi_j}^{(j)},\overline{\delta_{ki}}^{(j)}) \times \\
\Res\left(\overline{\psi_i}^{(i,j)},\overline{\psi_j}^{(i,j)}\right)^2
{\Rc_{i}^{(j)}(\Psi)}^2{\Rc_{j}^{(i)}(\Psi)}^2
\end{multline*}	
where the two irreducible polynomials  $\Rc_{j}^{(i)}(\Psi)$ and $\Rc_{i}^{(j)}(\Psi)$ are defined by the equalities
\begin{align*}
	\Res\left(\overline{\psi_i}^{(i)},\overline{\delta_{kj}}^{(i)},\overline{\delta_{jl}}^{(i)}\right) &= \Res\left(\overline{\psi_i}^{(i,j)},\overline{\psi_j}^{(i,j)}\right) \times \Rc_{j}^{(i)}(\Psi), \\
	 \Res\left(\overline{\psi_j}^{(j)},\overline{\delta_{ki}}^{(j)},\overline{\delta_{il}}^{(j)}\right) &= \Res\left(\overline{\psi_i}^{(i,j)},\overline{\psi_j}^{(i,j)}\right) \times \Rc_{i}^{(j)}(\Psi).
\end{align*}
\end{thm}

\begin{rem} We emphasize that, although $\delta_{pq}=-\delta_{qp}$ by definition, one can use these two minors indifferently in the above decomposition formula because the degree of partial homogeneity of discriminants is always an even number; see \eqref{eq:partialdegdisc}. This also holds for the other terms expressed as resultants because they always appear raised to the power 2. 
\end{rem}

\begin{proof} For the sake of simplicity, we prove this formula in the case $k=0$, $i=1$, $j=2$ and $l=3$; the other cases can be proved in a similar way.
% so we have to prove that
% \begin{multline*}
% \Disc(\delta_{01},\delta_{12},\delta_{23}) = \Delta_{4,d}(\Psi)
% \Disc(\overline{\psi_1}^{(1)},\overline{\delta_{23}}^{(1)})
% \Disc(\overline{\psi_2}^{(2)},\overline{\delta_{01}}^{(2)}) \times \\
% \Res\left(\overline{\psi_1}^{(1,2)},\overline{\psi_2}^{(1,2)}\right)^2
% {\Rc_{2}^{(1)}(\Psi)}^2{\Rc_{1}^{(2)}(\Psi)}^2
% \end{multline*}

To begin with, let $\Dc_{12}:=\Disc(\delta_{01},\delta_{12},\delta_{23})$ and $\Dc:=\Disc(\delta_{01},x_0\delta_{12},\delta_{23})$. Then, 
\begin{align*}
	\Dc &=(-1)^{d^3}\Dc_{12}\, \Disc(\overline{\delta_{01}}^{(0)},\overline{\delta_{23}}^{(0)})\Res(\overline{\delta_{01}}^{(0)},\overline{\delta_{12}}^{(0)},\overline{\delta_{23}}^{(0)})^2\\
	&= (-1)^{d^3}\Dc_{12}\, \Disc(x_1\overline{\psi_0}^{(0)},\overline{\delta_{23}}^{(0)})\Res(x_1\overline{\psi_0}^{(0)},\overline{\delta_{12}}^{(0)},\overline{\delta_{23}}^{(0)})^2\\
	&= (-1)^{d^3} \Dc_{12}\,\Disc(\overline{\delta_{23}}^{(0,1)})\Disc(\overline{\psi_0}^{(0)},\overline{\delta_{23}}^{(0)})\Res(\overline{\psi_0}^{(0,1)},\overline{\delta_{23}}^{(0,1)})^2 \times \\
	& \hspace{5cm}  \Res(\overline{\delta_{12}}^{(0,1)},\overline{\delta_{23}}^{(0,1)})^2\Res(\overline{\psi_0}^{(0)},\overline{\delta_{12}}^{(0)},\overline{\delta_{23}}^{(0)})^2\\
	&= (-1)^{d^3} \Dc_{12}\,\Disc(\overline{\delta_{23}}^{(0,1)})\Disc(\overline{\psi_0}^{(0)},\overline{\delta_{23}}^{(0)})\Res(\overline{\psi_0}^{(0,1)},\overline{\delta_{23}}^{(0,1)})^2 \times \\
	& \hspace{5cm}  \Res(x_2\overline{\psi_{1}}^{(0,1)},\overline{\delta_{23}}^{(0,1)})^2\Res(\overline{\psi_0}^{(0)},\overline{\delta_{12}}^{(0)},\overline{\delta_{23}}^{(0)})^2\\
	&=(-1)^{d^3} \Dc_{12}\,\Disc(\overline{\delta_{23}}^{(0,1)})\Disc(\overline{\psi_0}^{(0)},\overline{\delta_{23}}^{(0)})\Res(\overline{\psi_0}^{(0,1)},\overline{\delta_{23}}^{(0,1)})^2 \times \\
	& \hspace{5cm}  \Res(\overline{\delta_{23}}^{(0,1,2)})^2\Res(\overline{\psi_{1}}^{(0,1)},\overline{\delta_{23}}^{(0,1)})^2\Res(\overline{\psi_0}^{(0)},\overline{\delta_{12}}^{(0)},\overline{\delta_{23}}^{(0)})^2.
	% &= \Dc_{12}\,\Disc(\overline{\delta_{23}}^{(0,1)})\Disc(\overline{\psi_0}^{(0)},\overline{\delta_{23}}^{(0)})\Res(\overline{\psi_0}^{(0,1)},\overline{\delta_{23}}^{(0,1)})^2\\
	% & \hspace{5cm} \times \Res(x_3\overline{\psi_{2}}^{(0,1,2)})^2\Res(\overline{\psi_{1}}^{(0,1)},\overline{\delta_{23}}^{(0,1)})^2\Res(\overline{\psi_0}^{(0)},\overline{\delta_{12}}^{(0)},\overline{\delta_{23}}^{(0)})^2\\
\end{align*}
Now, using the relation $x_0\delta_{12}-x_1\delta_{02}+x_2\delta_{01}=0$, we get
$$\Dc=\Disc(\delta_{01},x_0\delta_{12},\delta_{23})=\Disc(\delta_{01},x_1\delta_{02},\delta_{23})$$
and hence, setting $\Dc_{02}:=\Disc(\delta_{01},\delta_{02},\delta_{23})$, similar computations as above yield the equalities
\begin{align*}
	\Dc &=(-1)^{d^3}\Dc_{02}\, 
	\Disc(\overline{\delta_{01}}^{(1)},\overline{\delta_{23}}^{(1)})
	\Res(\overline{\delta_{01}}^{(1)},\overline{\delta_{02}}^{(1)},\overline{\delta_{23}}^{(1)})^2\\	
	&= (-1)^{d^3}\Dc_{02}\,
	\Disc(x_0\overline{\psi_1}^{(1)},\overline{\delta_{23}}^{(1)})
	\Res(x_0\overline{\psi_1}^{(1)},\overline{\delta_{02}}^{(1)},\overline{\delta_{23}}^{(1)})^2\\
	&=(-1)^{d^3}\Dc_{02}\,
	\Disc(\overline{\delta_{23}}^{(0,1)})\Disc(\overline{\psi_1}^{(1)},\overline{\delta_{23}}^{(1)})
	\Res(\overline{\psi_1}^{(0,1)},\overline{\delta_{23}}^{(0,1)})^2  \times \\
	& \hspace{5cm}
	\Res(\overline{\delta_{02}}^{(0,1)},\overline{\delta_{23}}^{(0,1)})^2
	\Res(\overline{\psi_1}^{(1)},\overline{\delta_{02}}^{(1)},\overline{\delta_{23}}^{(1)})^2\\
	&=(-1)^{d^3}\Dc_{02}\,
	\Disc(\overline{\delta_{23}}^{(0,1)})\Disc(\overline{\psi_1}^{(1)},\overline{\delta_{23}}^{(1)})
	\Res(\overline{\psi_1}^{(0,1)},\overline{\delta_{23}}^{(0,1)})^2 \times \\
	& \hspace{5cm} 
	\Res(x_2\overline{\psi_{0}}^{(0,1)},\overline{\delta_{23}}^{(0,1)})^2
	\Res(\overline{\psi_1}^{(1)},\overline{\delta_{02}}^{(1)},\overline{\delta_{23}}^{(1)})^2\\
	&=(-1)^{d^3}\Dc_{02}\,
	\Disc(\overline{\delta_{23}}^{(0,1)})\Disc(\overline{\psi_1}^{(1)},\overline{\delta_{23}}^{(1)})
	\Res(\overline{\psi_1}^{(0,1)},\overline{\delta_{23}}^{(0,1)})^2 \times \\
	& \hspace{5cm} 
	\Res(\overline{\delta_{23}}^{(0,1,2)})^2
	\Res(\overline{\psi_{0}}^{(0,1)},\overline{\delta_{23}}^{(0,1)})^2
	\Res(\overline{\psi_1}^{(1)},\overline{\delta_{02}}^{(1)},\overline{\delta_{23}}^{(1)})^2.
\end{align*}
By comparing the two expressions we obtained for $\Dc$, and using the fact that all the above factors are nonzero, after simplification we deduce that
\begin{multline}\label{eq:D12D02}
\Dc_{12}\, \Disc(\overline{\psi_0}^{(0)},\overline{\delta_{23}}^{(0)})
\Res(\overline{\psi_0}^{(0)},\overline{\delta_{12}}^{(0)},\overline{\delta_{23}}^{(0)})^2= \\
	\Dc_{02}\, \Disc(\overline{\psi_1}^{(1)},\overline{\delta_{23}}^{(1)})
	\Res(\overline{\psi_1}^{(1)},\overline{\delta_{02}}^{(1)},\overline{\delta_{23}}^{(1)})^2.
\end{multline}

By proceeding as above, but using this time the relation $x_3\delta_{12}-x_2\delta_{13}+x_1\delta_{23}$ which induces the equality
$$ \Disc(\delta_{01},x_3\delta_{12},\delta_{23})=\Disc(\delta_{01},x_2\delta_{13},\delta_{23}),$$
we obtain the following similar equality:
\begin{multline}\label{eq:D12D13}
\Dc_{12}\, \Disc(\overline{\psi_3}^{(3)},\overline{\delta_{01}}^{(3)})
\Res(\overline{\psi_3}^{(3)},\overline{\delta_{01}}^{(3)},\overline{\delta_{12}}^{(3)})^2= \\ 
	\Dc_{13}\, \Disc(\overline{\psi_2}^{(2)},\overline{\delta_{01}}^{(2)})
	\Res(\overline{\psi_2}^{(2)},\overline{\delta_{01}}^{(2)},\overline{\delta_{13}}^{(2)})^2.
\end{multline}
Now, by Lemma \ref{lem:resIrr} there exist irreducible polynomials $\Rc_{2}^{(1)}$ and $\Rc_{1}^{(2)}$ such that 
\begin{align}\label{eq:R12R21}
\Res(\overline{\psi_1}^{(1)},\overline{\delta_{02}}^{(1)},\overline{\delta_{23}}^{(1)}) &= \Res(\overline{\psi_1}^{(1,2)},\overline{\psi_2}^{(1,2)}) \times \Rc_{2}^{(1)}, \\ \nonumber
\Res(\overline{\psi_2}^{(2)},\overline{\delta_{01}}^{(2)},\overline{\delta_{13}}^{(2)}) &= \Res(\overline{\psi_1}^{(1,2)},\overline{\psi_2}^{(1,2)}) \times \Rc_{1}^{(2)}.	
\end{align}
In addition, Lemma \ref{lem:gendisc} shows that both discriminants
$$ \Disc(\overline{\psi_2}^{(2)},\overline{\delta_{01}}^{(2)}) \textrm{ and } \Disc(\overline{\psi_1}^{(1)},\overline{\delta_{23}}^{(1)})$$
are irreducible and coprime (they do not depend on the same variables). So, from \eqref{eq:D12D02} and \eqref{eq:D12D13} we finally deduce that the product of irreducible and coprime polynomials
$$\Disc(\overline{\psi_2}^{(2)},\overline{\delta_{01}}^{(2)})\Disc(\overline{\psi_1}^{(1)},\overline{\delta_{23}}^{(1)})
\Res\left(\overline{\psi_1}^{(1,2)},\overline{\psi_2}^{(1,2)}\right)^2\left(\Rc_{2}^{(1)}\right)^2\left(\Rc_{1}^{(2)}\right)^2$$
divides $\Dc_{12}$ in $\UU$ (which is a UFD). The eigendiscriminant $\Delta_{4,d}$ is also another irreducible factor of this discriminant in $\UU$, and it is coprime with the other factors we  already identified. Let us list the degrees of all these factors:
\begin{itemize}
	\item[-] $\Dc_{12}$ is of degree $3d^2(4(d-1))=12(d-1)d^2$,
	\item[-] $\Delta_{4,d}$ is of degree $12(d-1)^3$,
	\item[-] $\Disc(\overline{\psi_1}^{(1)},\overline{\delta_{23}}^{(1)})$ is of degree 
	$(d-1)((d-2)+2(d-1))+d((d-1)+2(d-2))=6d^2-12d+4$,
	\item[-] $\Res(\overline{\psi_1}^{(1,2)},\overline{\psi_2}^{(1,2)})$ is of degree
	$2(d-1)$,
	\item[-] $\Rc_{1}^{(2)}$ and $\Rc_{2}^{(1)}$ are both of degree 
	$d^2+2(d-1)d - 2(d-1)=d^2+2(d-1)^2=3d^2-4d+2.$
\end{itemize}
By comparing these degrees, we deduce that there exists a nonzero integer $c$ such 
$$\Dc_{12} = c\, \Delta_{4,d}(\Psi) 
\Disc(\overline{\psi_1}^{(1)},\overline{\delta_{23}}^{(1)}) 
\Disc(\overline{\psi_2}^{(2)},\overline{\delta_{01}}^{(2)})
\Res\left(\overline{\psi_1}^{(1,2)},\overline{\psi_2}^{(1,2)}\right)^2
{\Rc_{2}^{(1)}(\Psi)}^2{\Rc_{1}^{(2)}(\Psi)}^2.$$

To conclude the proof we have to show that $c=\pm 1$. We observe that all the polynomials on the right-hand side of the above equation are primitive polynomials, so we have to prove that $\Dc_{12}$ is also a primitive polynomial. For this, we consider the specialization of $\psi_3$ to 0; $\Dc_{12}$ then specializes to
\begin{align}\label{eq:eq:prodspe}
\Disc(\delta_{01},\delta_{12},x_3\psi_2)&= \Disc(\overline{\delta_{01}}^{(3)},\overline{\delta_{12}}^{(3)})
				\Disc(\delta_{01},\delta_{12},\psi_2)
				\Res(\overline{\delta_{01}}^{(3)},\overline{\delta_{12}}^{(3)},\overline{\psi_2}^{(3)})^2.
\end{align}
In this product, the discriminant $\Disc(\overline{\delta_{01}}^{(3)},\overline{\delta_{12}}^{(3)})$ is a primitive polynomial as a consequence of Theorem \ref{thm:n=3}. The second discriminant can be further expanded as follows:
\begin{align}\label{eq:intermc}
\Disc(\delta_{01},\delta_{12},\psi_2) &= \Disc(\delta_{01},x_2\psi_1,\psi_2) \\
			&= \Disc(\overline{\delta_{01}}^{(2)},\overline{\psi_2}^{(2)})\Disc(\delta_{01},\psi_1,\psi_2)
			   \Res(\overline{\delta_{01}}^{(2)},\overline{\psi_1}^{(2)}, \overline{\psi_2}^{(2)})^2\\
   			&= \Disc(\overline{\delta_{01}}^{(2)},\overline{\psi_2}^{(2)})\Disc(\delta_{01},x_1\psi_0,\psi_2)
   			   \Res(x_1\overline{\psi_{0}}^{(2)},\overline{\psi_1}^{(2)}, \overline{\psi_2}^{(2)})^2.
\end{align}
The discriminant $\Disc(\overline{\delta_{01}}^{(2)},\overline{\psi_2}^{(2)})$ is a primitive polynomial because if $\psi_0$ is specialized to 0 then this discriminant specializes to
\begin{equation}\label{eq:discspeirr}
\Disc(x_0\overline{\psi_{1}}^{(2)},\overline{\psi_2}^{(2)})=(-1)^{(d-1)^2}\Disc(\overline{\psi_2}^{(0,2)})\Disc(\overline{\psi_{1}}^{(2)},\overline{\psi_2}^{(2)})\Res(\overline{\psi_1}^{(0,2)},\overline{\psi_2}^{(0,2)})^2	
\end{equation}
and all polynomials are primitive in this product. The discriminant $\Disc(\delta_{01},x_1\psi_0,\psi_2)$ is equal to 
\begin{multline*}
\Disc(\overline{\delta_{01}}^{(1)},\overline{\psi_{2}}^{(1)})\Disc(\delta_{01},\psi_0,\psi_2)\Res(\overline{\delta_{01}}^{(1)},\overline{\psi_{0}}^{(1)},\overline{\psi_{2}}^{(1)})^2 = \\
\Disc(\overline{\delta_{01}}^{(1)},\overline{\psi_{2}}^{(1)})\Disc(x_0\psi_1,\psi_0,\psi_2)\Res(x_0\overline{\psi_{1}}^{(1)},\overline{\psi_{0}}^{(1)},\overline{\psi_{2}}^{(1)})^2.
\end{multline*}
By developing further, we see that
$$
\Disc(x_0\psi_1,\psi_0,\psi_2) = (-1)^{d^3}
\Disc(\overline{\psi_0}^{(0)},\overline{\psi_2}^{(0)}) \Disc(\psi_1,\psi_0,\psi_2) \Res(\overline{\psi_0}^{(0)},\overline{\psi_1}^{(0)},\overline{\psi_2}^{(0)})^2	
$$
is a primitive polynomial, as well as
$$ 	
\Res(x_0\overline{\psi_{1}}^{(1)},\overline{\psi_{0}}^{(1)},\overline{\psi_{2}}^{(1)})= \Res(\overline{\psi_{0}}^{(0,1)},\overline{\psi_{2}}^{(0,1)})\Res(\overline{\psi_{1}}^{(1)},\overline{\psi_{0}}^{(1)},\overline{\psi_{2}}^{(1)}).
$$
The discriminant $\Disc(\overline{\delta_{01}}^{(1)},\overline{\psi_{2}}^{(1)})$ is also primitive by the same argument as the one used in \eqref{eq:discspeirr} and hence we have proved that $\Disc(\delta_{01},x_1\psi_0,\psi_2)$ is a primitive polynomial. In addition, the equality 
$$ 
\Res(x_1\overline{\psi_{0}}^{(2)},\overline{\psi_1}^{(2)}, \overline{\psi_2}^{(2)})=
\Res(\overline{\psi_{0}}^{(1,2)},\overline{\psi_1}^{(1,2)}, \overline{\psi_2}^{(1,2)})
\Res(\overline{\psi_{0}}^{(2)},\overline{\psi_1}^{(2)}, \overline{\psi_2}^{(2)})$$
shows that the resultant on the left-hand side of this equality is primitive. Therefore, coming back to  \eqref{eq:intermc} we finally deduce that $\Disc(\delta_{01},\delta_{12},\psi_2)$ is primitive.

To conclude the proof, i.e.~to prove that $\Dc_{12}$ is primitive, by $\eqref{eq:eq:prodspe}$ it remains to show that 
the resultant $\Res(\overline{\delta_{01}}^{(3)},\overline{\delta_{12}}^{(3)},\overline{\psi_2}^{(3)})$ is a primitive polynomial. By expanding it, we get
\begin{align}\label{eq:lasteq}
\Res(\overline{\delta_{01}}^{(3)},\overline{\delta_{12}}^{(3)},\overline{\psi_2}^{(3)}) &= 	
	\Res(\overline{\delta_{01}}^{(3)},x_2\overline{\psi_{1}}^{(3)},\overline{\psi_2}^{(3)})\\ \nonumber
	 &= \Res(\overline{\delta_{01}}^{(2,3)},\overline{\psi_2}^{(2,3)})
	    \Res(\overline{\delta_{01}}^{(3)},\overline{\psi_{1}}^{(3)},\overline{\psi_2}^{(3)})\\ \nonumber
	&= 	\Res(\overline{\delta_{01}}^{(2,3)},\overline{\psi_2}^{(2,3)})
	    \Res(\overline{\psi_{0}}^{(3)},\overline{\psi_{1}}^{(3)},\overline{\psi_2}^{(3)})
		\Res(\overline{\psi_{1}}^{(1,3)},\overline{\psi_2}^{(1,3)}).
\end{align}
The resultant $\Res(\overline{\delta_{01}}^{(2,3)},\overline{\psi_2}^{(2,3)})$ is a primitive polynomial because if we specialize $\psi_0$ to zero then it specializes to 
$$\Res(x_0\overline{\psi_{1}}^{(2,3)},\overline{\psi_2}^{(2,3)})=\psi_2(0,1,0,0)\Res(\overline{\psi_{1}}^{(2,3)},\overline{\psi_2}^{(2,3)}),$$
which is primitive, and the two other resultants on the right-hand side of \eqref{eq:lasteq} are also primitive polynomials. Therefore, $\Res(\overline{\delta_{01}}^{(3)},\overline{\delta_{12}}^{(3)},\overline{\psi_2}^{(3)})$ is a primitive polynomial and the theorem is proved.
\end{proof}

Similarly to Theorem \ref{thm:n=3}, Theorem \ref{thm:n=4} yields a universal formula which has interesting consequences for the eigendiscriminant. Thus, it allows to set the sign of $\Delta_{4,d}$ by relying on the normalizations of resultants and discriminants and we also get the following invariance property under the canonical action of $\mathrm{GL}_4$.

\begin{cor} Let $R$ be a commutative ring, $\Phi$ a rational map from $\PP^3_R$ to itself defined by polynomials of degree $d-1$ and $\varphi=[c_{i,j}]_{0\leq i,j\leq 3}$ be a matrix with entries in $R$. Then,   
	$$ \Delta_{4,d}(\Phi\circ \varphi)=\det(\varphi)^{(d-1)(3d^3-8d^2+8d-2)} \Delta_{4,d}(\Phi)$$
where the composition $\Phi\circ \varphi$ stands for the linear change of coordinates in $\PP^3_R$ defined by $\varphi$. 
\end{cor}

From a computational point of view, Theorem \ref{thm:n=4} allows to write down the eigendiscriminant $\Delta_{4,d}$ as a ratio of several determinants. Let us examine the formula we obtained with $k=0$, $i=1$, $j=2$, $l=3$:
\begin{multline*}
 \Disc(\delta_{01},\delta_{12},\delta_{23}) = \Delta_{4,d}(\Psi)
 \Disc(\overline{\psi_1}^{(1)},\overline{\delta_{23}}^{(1)})
 \Disc(\overline{\psi_2}^{(2)},\overline{\delta_{01}}^{(2)}) \times \\
 \Res\left(\overline{\psi_1}^{(1,2)},\overline{\psi_2}^{(1,2)}\right)^2
 {\Rc_{2}^{(1)}(\Psi)}^2{\Rc_{1}^{(2)}(\Psi)}^2.
\end{multline*}
Both discriminants $\Disc(\overline{\psi_1}^{(1)},\overline{\delta_{23}}^{(1)})$ and $\Disc(\overline{\psi_2}^{(2)},\overline{\delta_{01}}^{(2)})$
can be computed as the ratio of two determinants. Indeed, these discriminants are obtained from the defining equality \eqref{eq:defDisc} in which the two resultants can be computed as determinants (use the mixed resultant matrices given in \cite[3.11.19.25]{Jou97}). For its part, the resultant factor 
\begin{equation}\label{eq:extrafact}
\Res\left(\overline{\psi_1}^{(1,2)},\overline{\psi_2}^{(1,2)}\right)	
\end{equation}
is obtained as a single Sylvester determinant. The two terms ${\Rc_{2}^{(1)}(\Psi)}$ and ${\Rc_{1}^{(2)}(\Psi)}$ are not directly computed. Instead, one computes the two resultants that define them from \eqref{eq:R12R21} because the extraneous factor is equal to \eqref{eq:extrafact} which has already been computed. Notice that these two resultants are obtained as determinants of mixed resultant matrices (see loc.~cit.). Finally, 
it is necessary to compute the discriminant $\Disc(\delta_{01},\delta_{12},\delta_{23})$. Using \eqref{eq:defDisc}, this discriminant is the ratio of two resultants; the one in the numerator can be computed by means of the Macaulay formula \cite{DaDi01}, so it is the ratio of 2 determinants, and the one in the denominator can be computed again as the determinant of a mixed resultant matrix (see loc.~cit.). We notice that the computation of this discriminant is the main bottleneck in terms of computational complexity; it requires the computation of the resultant of four quaternary forms of degree $d,d,d$ and $3d-3$ respectively.

In summary, the eigendiscriminant $\Delta_{4,d}$ can be computed as the ratio of the product of 4 determinants by the product of 5  determinants, some of them being squared. Due to the size of these matrices and the intrinsic high degree of the eigendiscriminant, we were not able to compute it explicitly for some classes of cubic surfaces depending on parameters with a standard computer. Nevertheless, this formula allows to evaluate the eigendiscriminant as a polynomial over the integers for some specific instances of rational maps, which opens the door to interpolation methods.

\medskip

\subsection*{Acknowledgment} 
The author is grateful to Hirotachi Abo for stimulating discussions on eigendiscriminants.

% \bibliographystyle{plain}
% \bibliography{eigdisc}

\def\cprime{$'$}

\end{document}